\documentclass[11pt,reqno]{amsart}
\usepackage{amsfonts,amsfonts,amssymb}
\usepackage{cite}
\usepackage{color}
\usepackage{hyperref}
\usepackage{fleqn}
\usepackage{txfonts}

\theoremstyle{plain}
\newtheorem{theorem}{Theorem}[section]
\newtheorem{lemma}{Lemma}[section]
\theoremstyle{remark}
\newtheorem{remark}{Remark}[section]
\newtheorem{remarks}{Remarks}[section]

\numberwithin{equation}{section}

\newcommand{\e}{^\varepsilon}
\newcommand{\eps}{{\varepsilon}}
\newcommand{\ds}{\displaystyle}
\renewcommand{\a}{\alpha}
\renewcommand{\b}{\beta}
\newcommand{\D}{\mathrm{d}}

\newcommand{\cupl}{\bigcup\limits}
\newcommand{\suml}{\sum\limits}
\newcommand{\intl}{\int\limits}
\newcommand{\liml}{\lim\limits}
\newcommand{\minl}{\min\limits}
\newcommand{\A}{\mathcal{A}}
\renewcommand{\phi}{\varphi}

\begin{document}

\title[On the spectrum of narrow waveguides with periodic $\delta'$ traps]{On the spectrum of narrow Neumann waveguide with periodically distributed $\delta'$ traps}

\author[Pavel Exner]{Pavel Exner${^{1,2}}$}
\thanks{$^1$ Nuclear Physics Institute, Academy of Sciences of the Czech Republic, 
Hlavn\'{i} 130, 25068 \v{R}e\v{z} near Prague, Czech Republic; e-mail: exner@ujf.cas.cz}
\thanks{$^2$ Doppler Institute, Czech Technical University, B\v{r}ehov\'{a} 7, 11519 Prague, Czech Republic}

\author[Andrii Khrabustovskyi]{Andrii Khrabustovskyi${^{3,4}}$}
\thanks{$^3$ DFG Research Training Group 1294, Department of Mathematics, Karlsruhe 
Institute of Technology, Engesserstra{\ss}e 6,
76131 Karlsruhe, Germany; e-mail: andrii.khrabustovskyi@kit.edu}

\thanks{$^4$ Corresponding author}

\begin{abstract}
We analyze a family of singular Schr\"odinger operators describing a Neumann waveguide with a periodic array of singular traps of a $\delta'$ type. We show that in the limit when perpendicular size of the guide tends to zero and the $\delta'$ interactions are appropriately scaled, the first spectral gap is determined exclusively by geometric properties of the traps.
\end{abstract}

\keywords{periodic waveguides, Schr\"{o}dinger operators, $\delta'$ interaction, spectrum, gaps, asymptotic analysis\vspace{1mm}}

\maketitle

%=========================================================================
%=========Introduction====================================================
%=========================================================================

\section{\label{sec0} Introduction}

The problem addressed in this paper concerns the limiting behaviour of a particular class of Schr\"odinger operators with singular coefficients. They can be characterized as Neumann Laplacians on a cylindrical region in $\mathbb{R}^{n}$ perturbed by a array a singular ``traps'' consisting of a $\delta'$ interaction \cite{AGHH05} supported by the boundary of fixed subsets of the cylinder. We will be interested in the situation when the cylinder shrinks in the perpendicular direction and the parameter of the $\delta'$ interaction simultaneously changes making the latter weaker. We are going to show that the limiting behaviour of the first spectral gap is determined exclusively by geometric properties of the traps, namely their volume and surface area.

The motivation to study such an asymptotic behaviour is twofold. On one hand it is an interesting spectral problem falling within one of the traditional mathematical-physics categories, asking about relations between the geometric and spectral properties. On the other hand, it is of some practical interest, especially in the light of the recently growing interest to metamaterials and engineering of spectral properties. True, the $\delta'$ interactions with their peculiar scattering properties are rather a mathematical construct, however, they can be be approximated by regular or singular potentials following a seminal idea put forward in
\cite{CS98} and made mathematically rigorous in \cite{AN00, ENZ01}. Hence at least in the principal sense the result of this paper provides a way to achieve a prescribed spectral filtering.

%=========================================================================
%=========Setting & main result===========================================
%=========================================================================

\section{\label{sec1} Setting of the problem and main result}

\noindent In what follows $\eps>0$ will be a small parameter. For a fixed $n\in\mathbb{N}\setminus\{1\}$ we denote by $x'=(x_1,\dots,x_{n-1})$ and $x=(x',x_n)$  the Cartesian coordinates in $\mathbb{R}^{n-1}$ and $\mathbb{R}^{n}$, correspondingly. Let $\omega$ be an open domain in $\mathbb{R}^{n-1}$ with a Lipschitz boundary. By $\Omega\e$ we denote a straight cylinder in $\mathbb{R}^{n}$ with a cross-section $\eps\omega$, i.e.
 % ------------- %
$$\Omega\e=\left\{x=(x',x_n)\in\mathbb{R}^n:\ \eps^{-1}x'\in\omega\right\}.$$
 % ------------- %
Furthermore, we introduce the set
 % ------------- %
$$Y=\left\{x=(x',x_n)\in\mathbb{R}^n:\ |x_n|<1/2,\ x'\in\omega\right\},$$
 % ------------- %
which will play role of the period cell of the problem before scaling. Let $B$ be an arbitrary domain with a Lipschitz boundary $S=\partial B$ and such that $\overline{B}\subset Y$. For any $i\in\mathbb{Z}$ we denote
 % ------------- %
$$S_i\e=\eps(S+ie_n),\quad B_i\e=\eps(B+ie_n),\quad Y_i\e=\eps(Y+ie_n),$$
 % ------------- %
where $e_n=(0,0,\dots,0,1)$ is the unit vector along the cylinder axis.

Next we describe the family of operators $\mathcal{A}\e$ which will the main object of our interest in this paper. We denote $$\Gamma\e= \cupl_{i\in\mathbb{Z}}S_i\e$$ and introduce the sesquilinear form in the Hilbert space $L_2(\Omega\e)$  by
 % ------------- %
\begin{gather}\label{eta}
\eta\e[u,v]=\intl_{\Omega\e\setminus \Gamma\e}\nabla u\cdot\nabla\bar{v} \D x+a\e\suml_{i\in\mathbb{Z}}\intl_{S_i\e}(u_+-u_-)\overline{(v_+-v_-)}\,\D s,\quad a\e>0,
\end{gather}
 % ------------- %
here and in the following we denote by $u_+$ (respectively, $u_-$) the traces of the function $u$ taken from the exterior (respectively, interior) side of ${S_i\e}$. The form domain is supposed to be
 % ------------- %
\begin{multline*}
\mathrm{dom}(\eta\e)=H^1(\Omega\e\setminus \Gamma\e)\\ :=\left\{u\in L_2(\Omega\e):\ u\in H^1(\Omega\e\setminus\cupl_{i\in \mathbb{Z}} \overline{B_i\e}),\ u\in H^1(B_i\e)\text{ for all }i\in\mathbb{Z},\ \suml_{i\in\mathbb{Z}}\|\nabla u\|^2_{L_2(B_i\e)}<\infty\right\}.
\end{multline*}
 % ------------- %
The definition of $\eta\e[u,v]$ makes sense: the second sum in \eqref{eta} is finite because of the standard trace inequalities, namely
 % ------------- %
\begin{multline*}
\suml_{i\in\mathbb{Z}}\intl_{S_i\e}|u_+-u_-|^2\,\D s\leq
2\suml_{i\in\mathbb{Z}}\intl_{S_i\e}\left(|u_+|^2+|u_-|^2\right)\,\D s\\ \leq
C(\eps)\suml_{i\in\mathbb{Z}}\left(\|u\|^2_{H^1(Y_i\e\setminus \overline{B_i\e})}+\|u\|^2_{H^1(B_i\e)}\right)=
C(\eps)\|u\|_{H_1(\Omega\e\setminus \Gamma\e)}^2,
\end{multline*}
 % ------------- %
where the constant $C(\eps)$ is independent of $u$. Furthermore, it is straightforward to check that the form $\eta\e[u,v]$ is densely defined, closed and positive. Then (see, e.g., \cite[Chapter 6, Theorem 2.1]{K66}) there exists the unique self-adjoint and positive operator $\mathcal{A}\e$ associated with the form $\eta\e$, i.e.
 % ------------- %
\begin{gather}\label{eta-a}
(\mathcal{A}\e u,v)_{L_2(\Omega\e)}= \eta\e[u,v],\quad\forall u\in
\mathrm{dom}(\mathcal{A}\e),\ \forall  v\in \mathrm{dom}(\eta\e).
\end{gather}
 % ------------- %
If $u\in\mathrm{dom}(\mathcal{A}\e)$ and $u\in C^2(\Omega\e\setminus \Gamma\e)$ then via the integration by parts one can show easily that $(\mathcal{A}\e u) (x)=-\Delta u(x)$ for $x\in\Omega\e\setminus \Gamma\e$ and on the boundary $S_i\e$ one has
 % ------------- %
\begin{gather}\label{Scond}
\left({\partial u\over \partial n}\right)_+=\left({\partial u\over \partial n}\right)_-=:{\partial u\over \partial n},\quad {\partial u\over \partial n}=a\e(u_+-u_-),
\end{gather}
 % ------------- %
where $n$ is the outward-pointing unit normal to $S_i\e$. This makes it clear that the operators $\mathcal{A}\e$ have the meaning of Hamiltonians describing a waveguide with the Neumann outer boundary and an array of periodically spaced obstacles or traps given by a $\delta'$ interaction supported by $S_i\e$. Note that analogous Schr\"odinger operators in $\mathbb{R}^n$ with a $\delta'$ interaction supported by surfaces have been discussed recently in \cite{BLL13}.

By $\sigma(\mathcal{A}\e)$ we denote the spectrum of $\mathcal{A}\e$. Our goal in this paper is to describe its behavior as $\eps\to 0$ under the assumption that the coupling constant $a\e$ satisfies
 % ------------- %
\begin{gather}\label{a}
\liml_{\eps\to 0}{a\e\over\eps}=a>0.
\end{gather}
 % ------------- %

\begin{remark}\label{rem0}
A comment is due at this point to explain why we spoke in the introduction about a \emph{weak} $\delta'$ interaction. Comparing \eqref{Scond} with the standard definition of such an interaction \cite[Sec.~I.4]{AGHH05} we should regard rather the inverse $({a\e})^{-1}$ as the coupling parameter. On the other hand, due to peculiar properties of the interaction \cite[Theorem~I.4.3]{AGHH05} the $\delta'$ coupling is \emph{weak} if this quantity is \emph{large} satisfying, for instance, the asymptotic relation \eqref{a}.
\end{remark}

To state the result we shall use the notation $|\cdot|$ both for the volume of domain in $\mathbb{R}^n$ and for the area of $(n-1)$-dimensional surface in $\mathbb{R}^n$. Furthermore, we denote
 % ------------- %
$$\a={a|S|\over|B|},\quad \b={a|S|\over |B|}{|Y|\over |Y|-|B|}\,;$$
 % ------------- %
it is clear that $\a<\b$. Now we are in position to formulate the main result of this work.

 % ------------- %
\begin{theorem}\label{th1}
Let $L>0$ be an arbitrary number. Then the spectrum of the operator $\mathcal{A}\e$ in $[0,L]$ has under the assumptions stated above the following structure for $\eps$ small enough:
 % ------------- %
\begin{gather}\label{main1}
\sigma(\mathcal{A}\e)\cap [0,L]=[0,L]\setminus
(\a\e,\b\e),
\end{gather}
 % ------------- %
where the endpoints of the interval $(\a\e,\b\e)$ satisfy the relations
 % ------------- %
\begin{gather}\label{main2}
\liml_{\eps\to 0}\a\e=\a,\quad \liml_{\eps\to 0}\b\e=\b.
\end{gather}
 % ------------- %
\end{theorem}
 % ------------- %
The theorem will be proven in the next section. We postpone the outline of the proof to the remark preceding Lemma \ref{lm1} because we need to introduce first some more notations.

%=========================================================================
%=========Proof===========================================================
%=========================================================================

\section{\label{sec2} Proof of Theorem \ref{th1}.}

In what follows $C,C_1,\dots$ will be generic constants that do not depend on $\eps$. Let $D$ be an open domain in $\mathbb{R}^n$; by $\langle u \rangle_D$ we denote the normalized mean value of the function $u(x)$ in the domain $D$,
 % ------------- %
$$\langle u \rangle_D={1\over |D|}\intl_D u(x)\,\D x.$$
 % ------------- %
Furthermore, if $\Gamma\subset \mathbb{R}^n$ is an $(n-1)$-dimensional surface then the Euclidean metrics in $\mathbb{R}^n$ induces on $\Gamma$ the Riemannian metrics and measure. We denote by $\D s$ the density of this measure. Again by $\langle u\rangle_\Gamma$ we denote the normalized mean value of the function $u$ over $\Gamma$, i.e
 % ------------- %
$$\langle u\rangle_\Gamma=\ds{1\over |\Gamma|}\intl_{\Gamma}u \,\D s.$$
 % ------------- %
Next we introduce the following sets:

\begin{itemize}
\setlength{\itemsep}{3pt}

\item $Y\e=\eps Y$, the period cell,

\item $B\e=\eps B$, the trap,

\item $S\e=\eps S$, the trap boundary,

\item $F\e=Y\e\setminus\overline{B\e}$, the trap complement to the period cell,

\item $S_\pm\e=\left\{x=(x',x_n)\in\partial Y\e:\ x_n=\pm{\eps/2}\right\}$, the period cell ``lids''.

\end{itemize}

The Floquet-Bloch theory --- see, e.g., \cite{BHPW11,Ku93,RS78} --- establishes a relationship between the spectrum of $\A\e$ and the spectra of appropriate operators on $Y\e$. Specifically, for $\varphi\in [0,2\pi)$ we introduce the functional space $H_\varphi^1(Y\e\setminus S\e)$ consisting of functions from $H^1(Y\e\setminus S\e)$ that satisfy the following condition on the lateral parts of $\partial Y\e$:
 % ------------- %
\begin{gather}\label{theta1}
u|_{S_+\e}=\exp(i\varphi)T\e u|_{S_-\e},
\end{gather}
 % ------------- %
where $T\e:L_2(S_-\e)\to L_2(S_+\e)$, $(T\e f)(x)=f(x-\eps e_n)$, $u|_{S_\pm\e}$ are the traces of $u$ on $S_{\pm}\e$.

By $\eta_\varphi^{\eps}$ we denote the sesquilinear form defined by formula
 % ------------- %
\begin{gather}\label{eta+}
\eta_\varphi\e[u,v]=\intl_{Y\e\setminus S\e}\nabla u\cdot\nabla\bar{v} \D x+a\e\intl_{S\e}(u_+-u_-)\overline{(v_+-v_-)}\,\D s
\end{gather}
 % ------------- %
with the domain $H_\varphi^1(Y\e\setminus S\e)$. We define $\A_\varphi^{\eps}$ as the operator acting in $L_{2}(Y\e)$ being associated with the form $\eta_\varphi^{\eps}$:
 % ------------- %
\begin{gather*}
(\A_\varphi^{\eps} u,v)_{L_2(Y\e)}= \eta_\varphi^{\eps}[u,v],\quad\forall u\in
\mathrm{dom}(\A_\varphi^{\eps}),\  \forall v\in \mathrm{dom}(\eta_\varphi^{\eps}).
\end{gather*}
 % ------------- %
Since $Y\e\setminus S\e$ is compact, the operator $\A_\varphi^{\eps}$ has a purely discrete spectrum. We denote by $\left\{\lambda^{\phi}_{k} (\eps)\right\}_{k\in\mathbb{N}}$ the sequence of eigenvalues of $\A_\phi^{\eps}$ arranged in the increasing order and repeated according to their multiplicity.

According to the Floquet-Bloch theory one has the following representation:
 % ------------- %
\begin{gather}\label{repres1}
\sigma(\A\e)=\cupl_{k=1}^\infty \cupl_{\varphi\in [0,2\pi)}
\left\{\lambda^{\phi}_{k}(\eps)\right\}.
\end{gather}
 % ------------- %
Moreover, for any fixed $k\in\mathbb{N}$ the set $\cupl_{\varphi\in [0,2\pi)} \left\{\lambda^{\phi}_{k}(\eps)\right\}$, in other words, the $k$th spectral band, is a compact interval.

We also introduce the operators $\mathcal{A}\e_N$ and $\mathcal{A}\e_D$, which are defined in a similar way as $\mathcal{A}\e_\varphi$, however, with \eqref{theta1} replaced by Neumann and Dirichlet boundary conditions of $S\e_{\pm}$, respectively. More precisely, we denote by $\eta^{\eps}_N$ (correspondingly, $\eta_D^{\eps}$) the sesquilinear forms in $L_2(Y\e)$ defined by (\ref{eta+}) and the domain $H^1(Y\e\setminus S\e)$ (correspondingly, $\widehat{H}^1_0(Y\e\setminus S\e)=\left\{u\in H^1(Y\e\setminus S\e):\ u=0\text{ on }S_{+}\e\cup S_-\e\right\}$). The above indicated operators are then associated with these forms, $\eta_N^{\eps}$ and $\eta_D^{\eps}$, respectively, i.e.
 % ------------- %
\begin{gather*}
(\A_*^{\eps} u,v)_{L_2(Y\e)}= \eta_*^{\eps}[u,v],\quad\forall u\in
\mathrm{dom}(\A_*^{\eps}),\ \forall v\in \mathrm{dom}(\eta_*^{\eps}),
\end{gather*}
 % ------------- %
where $*$ is $N$ (correspondingly, $D$).

As with $\A_\varphi^{\eps}$, the spectra of the operators $\A_N^{\eps}$ and $\A_D^{\eps}$ are purely discrete. We denote by $\left\{\lambda_k^N (\eps)\right\}_{k\in\mathbb{N}}$ (correspondingly, $\left\{\lambda_k^D(\eps)\right\}_{k\in\mathbb{N}}$) the sequence of eigenvalues of $\A_N^{\eps}$ (correspondingly, of $\A_D^{\eps}$) arranged in the ascending order and repeated according to their multiplicity.

From the min-max principle --- see, e.g., \cite[Chapter XIII]{RS78} --- and the inclusions
 % ------------- %
$$H^1(Y\e\setminus S\e) \supset H^1_\varphi (Y\e\setminus S\e)\supset \widehat{H}^1_0(Y\e\setminus S\e)$$
 % ------------- %
we infer that
 % ------------- %
 \begin{gather}\label{enclosure}
\forall k\in \mathbb{N},\ \forall\varphi\in [0,2\pi):\quad
\lambda_k^N(\eps) \leq \lambda_k^\varphi(\eps) \leq
\lambda_k^D(\eps).
\end{gather}
 % ------------- %

\begin{remarks}\label{rem1a} (a) With these preliminaries, we are able to provide the promised brief description of the proof of Theorem \ref{th1}. It is clear that the left edge of the first spectral band of $\mathcal{A}\e$ coincides with zero, while the right one is situated between the first antiperiodic eigenvalue $\lambda_{1}^\varphi(\eps)$, $\varphi=\pi$, and the first Dirichlet eigenvalue $\lambda_{1}^D(\eps)$. We are going to prove (see Lemmata \ref{lm2} and \ref{lm4} below) that they both converge to $\alpha$ as $\eps\to 0$. Similarly we can localize the left edge of the second spectral band between the second Neumann eigenvalue $\lambda_{2}^N(\eps)$ and the second periodic eigenvalue $\lambda_{2}^\varphi(\eps)$, $\varphi=0$, of which we will prove (see Lemmata \ref{lm3} and \ref{lm4} below) that they both converge to $\beta$ as $\eps\to 0$. Finally, we intend to  prove that $\lambda_{2}^\varphi(\eps)$, $\varphi\not= 0$, converges to infinity as $\eps\to 0$ which means that the right edge of the second spectral band exceeds any fixed $L$ provided $\eps$ is small enough. These results taken together constitute the claim of Theorem~\ref{th1}.

\smallskip

\label{rem1b} \noindent (b) We stress that the band edges need not in general coincide with the corresponding periodic (antiperiodic) solutions even if the system exhibits periodicity in one direction only \cite{HKSW07, EKW10}. What matters is that we can squeeze them between two values which converge to the same limit as $\eps\to 0$.
\end{remarks}

 % ------------- %
\begin{lemma}\label{lm1}
Let $\varphi\not= 0.$ In the limit $\eps\to 0$ one has
 % ------------- %
\begin{gather}\label{Phi2}
\lambda_2^\varphi(\eps)\to \infty.
\end{gather}
 % ------------- %
\end{lemma}

\begin{proof}
We denote
\begin{itemize}
\setlength{\itemsep}{3pt}

\item $F=\eps^{-1}F\e$, the scaled trap complement to the scaled period cell,

\item $S_{\pm}=\eps^{-1}S\e_\pm$, the scaled period cell ``lids''.

\end{itemize}

We introduce the sesquilinear form $\tilde{\eta}\e_{\varphi}$ in the space $L_2(Y)$ defined by the formula
 % ------------- %
\begin{gather*}
\tilde{\eta}_\varphi\e[u,v]=\intl_{Y \setminus S}\nabla u\cdot\nabla\bar{v} \,\D x+a\e\eps \intl_{S}(u_+-u_-)\overline{(v_+-v_-)}\,\D s,
\end{gather*}
 % ------------- %
with
 % ------------- %
\begin{gather*}
\mathrm{dom}(\tilde{\eta}_{\varphi}\e)=\left\{u\in H^1(Y\setminus S):\ u|_{S_+}=\exp(i\varphi)T u|_{S_-}\right\},
\end{gather*}
 % ------------- %
where $T:L_2(S_-)\to L_2(S_+)$, $\:(T f)(x)=f(x- e_n)$, and $u|_{S_\pm}$ are the traces of $u$ on $S_{\pm}$. Let $\tilde{\mathcal{A}}\e_\varphi$ be the operator associated with this form and let $\left\{\tilde{\lambda}^{\varphi}_{k}(\eps)\right\}_{k\in\mathbb{N}}$ be the sequence of its eigenvalues arranged in the ascending order and repeated according to their multiplicity. It is easy to see that for all $k\in\mathbb{N}$
 % ------------- %
\begin{gather}\label{lambda-lambda}
\tilde{\lambda}^{\varphi}_{k}(\eps)=\eps^{2}{\lambda}^{\varphi}_{k}(\eps).
\end{gather}
 % ------------- %
In the same Hilbert space $L_2(Y)$ we introduce the sesquilinear form $\tilde{\eta}_{\varphi}$ by the formula
 % ------------- %
\begin{gather*}
\tilde{\eta}_\varphi[u,v]=\intl_{Y \setminus S}\nabla u\cdot\nabla\bar{v}\, \D x
\end{gather*}
 % ------------- %
with $\mathrm{dom}(\tilde{\eta}_\varphi)=\mathrm{dom}(\tilde{\eta}_\varphi\e)$, and by $\tilde{\mathcal{A}}_\varphi$ we denote the operator generated by this form, with $\left\{\tilde{\lambda}^{\varphi}_{k}\right\}_{k\in\mathbb{N}}$ being the sequence of eigenvalues of $\tilde{\mathcal{A}}_\varphi$ written in the increasing order and repeated according to their multiplicity. It is clear from the definition that
 % ------------- %
$$\tilde{\mathcal{A}}_\varphi=-\Delta_\varphi(F)\oplus -\Delta(B),$$
 % ------------- %
where $\Delta_\varphi(F)$ (respectively, $\Delta(B)$) is the Laplacian in $L_2(F)$ (respectively, $L_2(B)$) with the Neumann boundary conditions on $\partial F\setminus (S_-\cup S_+)$  and $\varphi$-periodic boundary conditions on $S_-\cup S_+$ (respectively, with the Neumann boundary condition on $S$). One can check easily that for any $\varphi\not=0 $ we have
 % ------------- %
\begin{gather}\label{lambda_bf_prop}
\tilde{\lambda}_2^\varphi>0.
\end{gather}
 % ------------- %
To conclude the argument we are going to demonstrate that
 % ------------- %
\begin{gather}\label{lambda_bf}
\forall k\in\mathbb{N}:\  \tilde{\lambda}^{\varphi}_{k}(\eps)\to \tilde{\lambda}^{\varphi}_{k}
\end{gather}
 % ------------- %
holds, then \eqref{Phi2} will follow directly from \eqref{lambda-lambda}--\eqref{lambda_bf}. It remains therefore to prove \eqref{lambda_bf}. We denote
 % ------------- %
$$\tilde{L}_\varphi\e=(\tilde{\mathcal{A}}_\varphi\e+I)^{-1},\quad \tilde{L}_\varphi=(\tilde{\mathcal{A}}_\varphi+I)^{-1},$$
 % ------------- %
where $I$ is the identity operator. The operators $\tilde{L}\e_\varphi$, $\tilde{L}_\varphi$ are compact and positive (thus self-adjoint), and furthermore
 % ------------- %
\begin{gather}\label{c2}
\|\tilde{L}_\varphi\e\|\leq 1.
\end{gather}
 % ------------- %
Next we want to prove that
 % ------------- %
\begin{gather}\label{c3}
\forall f\in L_2(Y):\ \tilde{L}_\varphi\e f\to\tilde{L}_\varphi f\;\text{ in } L_2(Y)\;\text{ as }\;\eps\to 0.
\end{gather}
 % ------------- %
We take an arbitrary $f\in L_2(Y)$ and set $u\e=\tilde{L}_\varphi\e f$.  It is clear that
 % ------------- %
$$\tilde{\eta}_\varphi\e[u\e,u\e]+\|u\e\|_{L_2(Y)}\leq C,$$
 % ------------- %
which, in particular, implies that the functions $u\e$ are bounded in $H^1(Y\setminus S)$ uniformly in $\eps$. Therefore by the Rellich-Kondrachov embedding theorem there exist $u\in H^1(Y\setminus S)$ and a subsequence $\eps_k$, $\:k=1,2,3\dots$, satisfying $\eps_k\searrow 0$ as $\eps\to 0$ such that
 % ------------- %
\begin{gather}\label{conv1bf}
u\e\rightharpoonup u\;\text{ in }H^1(Y\setminus S)\;\text{ as }\;\eps=\eps_k\to 0,\\
\label{conv1+bf}
u\e\to u\;\text{ in }L_2(Y)\;\text{ as }\;\eps=\eps_k\to 0.
\end{gather}
 % ------------- %
Furthermore, in view of the trace theorem
 % ------------- %
\begin{gather}\label{conv2bf}
u\e\to u\;\text{ in }L_2(\partial Y\cup S)\;\text{ as }\;\eps=\eps_k\to 0,
\end{gather}
 % ------------- %
and consequently, $u\in \mathrm{dom}(\tilde{\eta}_\varphi)$. Given an arbitrary $v\in \mathrm{dom}(\tilde{\eta}_\varphi\e)$, we have the following identity:
 % ------------- %
\begin{gather}\label{int_in_bf}
\intl_{Y \setminus S}\nabla u\e\cdot\nabla\bar{v}\,\D x+a\e\eps \intl_{S}(u\e_+-u\e_-)\overline{(v_+-v_-)}\D s+\intl_{Y}u\e \bar{v}\D x= \intl_{Y}f\bar{v}\,\D x.
\end{gather}
 % ------------- %
Using  \eqref{conv1bf}--\eqref{conv2bf} and keeping in mind that $a\e\eps\to 0$ holds as $\eps\to 0$ we pass to the limit in \eqref{int_in_bf} as $\eps=\eps_k\to 0$ and obtain
 % ------------- %
\begin{gather}\label{int_ineq_bf_final}
\intl_{Y \setminus S}\nabla u\cdot\nabla\bar{v} \,\D x+\intl_{Y}u \bar{v}\D x= \intl_{Y}f\bar{v}\,\D x,\quad \forall v\in \mathrm{dom}(\tilde{\eta}_\varphi).
\end{gather}
 % ------------- %
It follows from \eqref{int_ineq_bf_final} that $u=\tilde{L}_\varphi f$. Since $u$ is independent of the subsequence $\eps_k$, the whole sequence $u\e$ converges to $u$ as $\eps\to 0$, in other words, \eqref{c3} holds true.

Finally, again using the Rellich-Kondrachev embedding theorem, we conclude that for an arbitrary sequence $f\e$, which is bounded in $L_2(Y)$ uniformly in $\eps$, there exist $w\in L_2(Y)$ and a subsequence $\eps_k$, $\:k=1,2,3\dots$, with $\eps_k\underset{k\to\infty}\searrow 0$ such that
 % ------------- %
\begin{gather}\label{c4}
\tilde{L}_\varphi\e f\e\to w\;\text{ in }L_2(Y)\;\text{ as }\;\eps=\eps_k\to 0.
\end{gather}
 % ------------- %
We denote by $\left\{\tilde{\mu}_{k}^\varphi(\eps)\right\}_{k\in\mathbb{N}}$ and $\left\{\tilde{\mu}_{k}^\varphi\right\}_{k\in\mathbb{N}}$ the
sequences of eigenvalues of $\tilde{L}_\varphi^{\eps}$ and $\tilde{L}_\varphi$, respectively, written in the decreasing order and repeated according to their multiplicity. According to Lemma~1 from \cite{IOS89} it follows from \eqref{c2}, \eqref{c3}, and \eqref{c4} that
 % ------------- %
\begin{gather}\label{mu}
\forall k\in\mathbb{N}:\ \tilde{\mu}^\varphi_k(\eps)\to\tilde{\mu}^\varphi_k\;\text{ as }\;\eps\to 0.
\end{gather}
 % ------------- %
However, since $\tilde{\lambda}_k^\varphi(\eps)={1\over \tilde{\mu}^\varphi_k(\eps)}-1$ and  $\tilde{\lambda}_k^\varphi={1\over \tilde{\mu}^\varphi_k}-1$,
\eqref{mu} implies the sought relation \eqref{lambda_bf} concluding thus the proof of the lemma.
\end{proof}

Now we have to inspect the behaviour of the first Dirichlet eigenvalue, which we use to estimate the left gap edge from above, in the the limit $\eps\to 0$.

\begin{lemma}\label{lm2}
One has
 % ------------- %
\begin{gather}\label{convD}
\lambda_1^D(\eps)\to \alpha\;\text{ as }\;\eps\to 0.
\end{gather}
 % ------------- %
\end{lemma}

\begin{proof}
Let $u_D\e$ be the eigenfunction of $\A\e_D$, which corresponds to $\lambda_1^D(\eps)$, determined uniquely by the following requirements:
 % ------------- %
\begin{gather}\label{normalizationD}
\|u_D\e\|_{L_2(Y\e)}=1,\\\label{normalizationD+}
u\e_D\;\text{ is real-valued},\quad \intl_{B\e} u\e_D \,\D x\geq 0.
\end{gather}
 % ------------- %
From the min-max principle we get
 % ------------- %
\begin{gather}\label{minmaxD}
\lambda^D_1(\eps)=\eta_D\e[u\e_D,u_D\e]=\minl_{u\in \mathcal{H}_D\e}\eta_D\e[u,u],
\end{gather}
 % ------------- %
where $\mathcal{H}_D\e=\left\{u\in\mathrm{dom}(\eta_D\e):\ \|u\|_{L_2(Y\e)}=1\right\}$. We construct an approximation, denoted as $v\e_D$, to the eigenfunction $u\e_D$ using the formula
 % ------------- %
\begin{gather}\label{vD}
v\e_D(x)=\begin{cases}0,&x\in F\e,\\|B\e|^{-1/2},&x\in B\e.\end{cases}
\end{gather}
 % ------------- %
It is clear that $v\e_D\in\mathrm{dom}(\eta_D^\eps)$. Taking into account that $\nabla v\e_D=0$ holds in $B\e\cup F\e$ we get
 % ------------- %
\begin{gather}\label{vD_est1}
\eta_D^\eps[v\e_D,v\e_D]={a\e\intl_{S\e}((v\e_D)_+-(v\e_D)_-)^2 \D x}={a\e |S\e|\over |B\e|}={a\e |S|\over \eps|B|}\sim \a\;\text{ as }\;\eps\to 0,\\\label{vD_est2}
\|v\e_D\|_{L_2(Y\e)}=1,
\end{gather}
 % ------------- %
and therefore $v_D\e\in \mathcal{H}\e_D$. Consequently, in view of \eqref{minmaxD} we can infer that
 % ------------- %
\begin{gather}\label{minmaxD1}
\eta_D^\eps[u\e_D,u\e_D] \leq {\eta_D^\eps[v\e_D,v\e_D]}.
\end{gather}
 % ------------- %
If follows from \eqref{vD_est1} and \eqref{minmaxD1} that
 % ------------- %
\begin{gather}\label{lambdaD_est1}
\eta_D^\eps[u\e_D,u\e_D]\leq C,
\end{gather}
 % ------------- %
and using \eqref{lambdaD_est1} one arrives at the following Friedrichs- and Poincar\'{e}-type inequalities:
 % ------------- %
\begin{gather}\label{uD_est3}
\|u\e_D\|^2_{L_2(F\e)}\leq C\eps^2\|\nabla u\e_D\|^2_{L_2(F\e)}\leq C\eps^2\eta_D^\eps[u\e_D,u\e_D]\leq C_1\eps^2,\\\label{uD_est4}
\|u\e_D-\langle u\e_D \rangle_{B\e}\|^2_{L_2(B\e)}\leq C\eps^2\|\nabla u\e_D\|^2_{L_2(B\e)}\leq C\eps^2\eta\e_D[u\e_D,u\e_D]\leq C_1\eps^2,
\end{gather}
 % ------------- %
where $\langle\cdot\rangle$ as usual denotes the mean value. Moreover, using \eqref{normalizationD}, \eqref{uD_est3}, and \eqref{uD_est4} we get
 % ------------- %
\begin{gather*}
\left|\langle u\e_D \rangle_{B\e}\right|^2|B\e|=1-\|u\e_D\|^2_{L_2(F\e)}-\|u\e_D-\langle u\e_D \rangle_{B\e}\|^2_{L_2(B\e)}=1+\mathcal{O}(\eps^2),
\end{gather*}
 % ------------- %
which implies, taking into account \eqref{normalizationD+}, that
 % ------------- %
\begin{gather}\label{uD_est5}
\left(\langle u\e_D \rangle_{B\e}|B\e|^{1/2}-1\right)\to 0\;\text{ as }\;\eps\to 0.
\end{gather}
 % ------------- %
Finally, using \eqref{uD_est4} and \eqref{uD_est5} we conclude that
 % ------------- %
\begin{multline}\label{uD_est6}
\|u\e_D-|B\e|^{-1/2}\|^2_{L_2(B\e)}\leq 2\|u\e_D-\langle u\e_D \rangle_{B\e}\|^2_{L_2(B\e)}\\+2|B\e|\left|\langle u\e_D \rangle_{B\e}-|B\e|^{-1/2}\right|^2\to 0\;\text{ as }\;\eps\to 0.
\end{multline}
 % ------------- %

In the next step we represent the eigenfunction $u\e_D$ in the form of a sum,
 % ------------- %
\begin{gather}\label{represD}
u\e_D=v\e_D+w\e_D,
\end{gather}
 % ------------- %
and estimate the remainder $w\e_D$. Plugging \eqref{represD} into \eqref{minmaxD1} we obtain
 % ------------- %
\begin{multline}\label{Destim1}
\eta\e_D[w_D\e,w_D\e]\leq -2\eta\e_D [v_D\e,w_D\e]\\=
-2a\e\intl_{S\e}\left((v\e_D)_+-(v\e_D)_-\right)\left((w\e_D)_+-(w\e_D)_-\right)\,\D s.
\end{multline}
 % ------------- %
Let us recall the following standard trace inequalities:
 % ------------- %
$$\forall u\in H^1(F),v\in H^1(B):\quad \|u\|_{L_2(S)}^2\leq C\|u\|^2_{H^1(F)},\ \|v\|_{L_2(S)}^2\leq C\|v\|^2_{H^1(B)}.$$
 % ------------- %
Using them together with a change of variables, $x\mapsto x\eps$, one can derive the estimates
 % ------------- %
\begin{gather}\label{uD_est7}
\intl_{S\e}\left|(u\e_D)_+\right|^2\D s\leq C\eps^{-1}\left(\|u\e_D\|^2_{L_2(F\e)}+\eps^2\|\nabla u\e_D\|^2_{L_2(F\e)}\right),\\ \label{uD_est8}
\intl_{S\e}\left|(u\e_D)_--|B\e|^{-1/2}\right|^2\D s\leq C\eps^{-1}\left(\|u\e_D-|B\e|^{-1/2}\|^2_{L_2(B\e)}+\eps^2\|\nabla u\e_D\|^2_{L_2(B\e)}\right).
\end{gather}
 % ------------- %
Then taking into account \eqref{a} and \eqref{uD_est7}--\eqref{uD_est8} we obtain from \eqref{Destim1} the inequalities
 % ------------- %
\begin{multline*}
\left|\eta\e_D[w_D\e,w_D\e]\right|^2\leq 4(a\e)^2\intl_{S\e}\left|(v\e_D)_+-(v\e_D)_-\right|^2\D s\cdot \intl_{S\e}\left|(w\e_D)_+-(w\e_D)_-\right|^2\D s\\\leq 8 {(a\e)^2|S\e|\over |B\e|}\intl_{S\e}\left(\left|(u\e_D)_+\right|^2+\left|(u\e_D)_--|B\e|^{-1/2}\right|^2\right)\,\D s\\\leq C\left(\|u_D\e\|^2_{L_2(F\e)}+\|u\e_D-|B\e|^{-1/2}\|^2_{L_2(B\e)}+\eps^2\|\nabla u\e_D\|^2_{L_2(B\e\cup F\e)}\right),
\end{multline*}
 % ------------- %
thus in view of \eqref{lambdaD_est1}, \eqref{uD_est3}, and \eqref{uD_est6} we find
 % ------------- %
\begin{gather}\label{wD_est1}
\eta\e_D[w_D\e,w_D\e]\to 0\;\text{ as }\;\eps\to 0.
\end{gather}
 % ------------- %
It follows from \eqref{vD_est1}, \eqref{represD}, and \eqref{wD_est1} that
 % ------------- %
\begin{gather}\label{finalD}\lambda_1^D(\eps)=\eta\e_D[u_D\e,u_D\e]\sim \eta\e_D[v_D\e,v_D\e]\sim\alpha\;\text{ as }\;\eps\to 0,
\end{gather}
 % ------------- %
which is the claim we have set up to prove.
\end{proof}

Next we need an analogous claim for the second Neumann eigenvalue which we use to estimate the right gap edge from below.

\begin{lemma}\label{lm3}
One has
 % ------------- %
\begin{gather}\label{convN}
\lambda_2^N(\eps)\to \beta\;\text{ as }\;\eps\to 0.
\end{gather}
 % ------------- %
\end{lemma}

\begin{proof}
Let $u_N\e$ be the eigenfunction of $\A\e_N$, which corresponds to $\lambda_2^N(\eps)$ and satisfying the following requirement:
 % ------------- %
\begin{gather}\label{normalizationN}
\|u_N\e\|_{L_2(Y\e)}=1,\\\label{normalizationN+} u\e_N\text{ is real-valued},\quad \intl_{B\e} u\e_N \,\D x\geq 0.
\end{gather}
 % ------------- %
Since $\lambda_1^N(\eps)=0$ and the corresponding eigenspace consists of constant functions, we have
 % ------------- %
\begin{gather}\label{normalizationN++}
\intl_{Y\e}u_N\e \,\D x=0,
\end{gather}
 % ------------- %
and in view of the min-max principle it holds
 % ------------- %
\begin{gather}\label{minmaxN}
\lambda^N_2(\eps)=\eta_N\e[u\e_N,u_N\e]=\minl_{u\in \mathcal{H}_N\e}\eta_N\e[u,u],
\end{gather}
 % ------------- %
where $\mathcal{H}_N\e=\left\{u\in\mathrm{dom}(\eta_N\e):\ \|u\|_{L_2(Y\e)}=1,\quad \intl_{Y\e}u \,\D x=0\right\}$. As in the previous proof, we construct an approximation $v\e_N$ to the eigenfunction $u\e_N$, this time defined by the formula
 % ------------- %
\begin{gather}\label{vN}
v\e_N(x)=\begin{cases}\kappa_F\e:= -\sqrt{|B\e|\over |F\e|\cdot |Y\e|},&x\in F\e,\\\kappa_B\e:=\sqrt{|F\e|\over |B\e|\cdot |Y\e|},&x\in B\e.\end{cases}
\end{gather}
 % ------------- %
Obviously $v\e_N\in\mathrm{dom}(\eta_N^\eps)$ and it is easy to see that
 % ------------- %
\begin{gather}\label{vN_est1}
\eta_N^\eps[v\e_N,v\e_N]={a\e |S\e|\left(\kappa\e_F-\kappa\e_B\right)^2}={a\e |S|\cdot|Y|\over \eps|B|\cdot|F|}\sim \b\;\text{ as }\;\eps\to 0,\\\label{vN_est2}
\|v\e_N\|_{L_2(Y\e)}=1,\\\label{vN_est3}\intl_{Y\e}v_N\e \,\D x=0.
\end{gather}
 % ------------- %
In view of \eqref{vN_est2} and \eqref{vN_est3}, $\,v_N\e$ belongs to $\mathcal{H}_N\e$, and therefore using \eqref{minmaxN} and taking into account \eqref{vN_est1} we get
 % ------------- %
\begin{gather}\label{lambdaN_est1}
{\eta_N^\eps[u\e_N,u\e_N]}\leq {\eta_N^\eps[v\e_N,v\e_N]}\leq  C.
\end{gather}
 % ------------- %
Using \eqref{lambdaN_est1} one can derive the following Poincar\'{e}-type inequalities:
 % ------------- %
\begin{gather}\label{uN_est1}
\|u\e_N-\langle u\e_N \rangle_{F\e}\|^2_{L_2(F\e)}\leq C\eps^2\|\nabla u\e_N\|^2_{L_2(F\e)}\leq C\eps^2\eta_N^\eps[u\e_N,u\e_N]\leq C_1\eps^2,\\\label{uN_est2}
\|u\e_N-\langle u\e_N \rangle_{B\e}\|^2_{L_2(B\e)}\leq C\eps^2\|\nabla u\e_N\|^2_{L_2(B\e)}\leq C\eps^2\eta_N^\eps[u\e_N,u\e_N]\leq C_1\eps^2.
\end{gather}
 % ------------- %
Moreover, using \eqref{normalizationN}, \eqref{normalizationN++}, \eqref{uN_est1}, and \eqref{uN_est2} we infer that
 % ------------- %
\begin{multline*}
\left|\langle u\e_N \rangle_{F\e}\right|^2|F\e|+\left|\langle u\e_N \rangle_{B\e}\right|^2|B\e|\\=1-\|u\e_N-\langle u\e_N \rangle_{F\e}\|^2_{L_2(F\e)}+\|u\e_N-\langle u\e_N \rangle_{B\e}\|^2_{L_2(B\e)}=1+\mathcal{O}(\eps^2),\end{multline*}
$$\langle u\e_N \rangle_{F\e}|F\e|+\langle u\e_N \rangle_{B\e}|B\e|=\intl_{Y\e}u_N\e \,\D x=0,$$
 % ------------- %
from where, taking into account \eqref{normalizationN+}, we obtain the following asymptotics:
 % ------------- %
\begin{gather}\label{uN_est3}
\langle u\e_N \rangle_{F\e}= \kappa_F\e(1+\mathcal{O}(\eps^2))^{1/2},\ \langle u\e_N \rangle_{B\e}= \kappa_B\e(1+\mathcal{O}(\eps^2))^{1/2}\;\text{ as }\;\eps\to 0.
\end{gather}
 % ------------- %
Finally, using \eqref{uN_est1}--\eqref{uN_est3} one can easily arrive at the relation
 % ------------- %
\begin{gather}\label{uN_est4}
\|u\e_N-\kappa_F\e\|^2_{L_2(F\e)}+\|u\e_N-\kappa_B\e\|^2_{L_2(B\e)}\to 0\;\text{ as }\;\eps\to 0.
\end{gather}
 % ------------- %
In the next step, we again represent the eigenfuction $u\e_N$ in the form of a sum,
 % ------------- %
\begin{gather}\label{represN}
u\e_N=v\e_N+w\e_N
\end{gather}
 % ------------- %
and estimate the remainder $w\e_N$. We plug \eqref{represN} into \eqref{lambdaN_est1} and obtain
 % ------------- %
\begin{multline}\label{Nestim1}
\eta\e_N[w_N\e,w_N\e]\leq -2\eta\e_N [v_N\e,w_N\e]\\=
-2a\e\intl_{S\e}\left((v\e_N)_+-(v\e_N)_-\right)\left((w\e_N)_+-(w\e_N)_-\right)\,\D s.
\end{multline}
 % ------------- %
In analogy with relation (\ref{uD_est8}) in the proof of Lemma~\ref{lm2} we employ the following trace inequalities:
 % ------------- %
\begin{gather*}
\intl_{S\e}\left|(u\e_N)_+-\kappa_F\e\right|^2\D s\leq C\eps^{-1}\left(\|u\e_N-\kappa_F\e\|^2_{L_2(F\e)}+\eps^2\|\nabla u\e_N\|^2_{L_2(F\e)}\right),\\
\intl_{S\e}\left|(u\e_N)_--\kappa_B\e\right|^2\D s\leq C\eps^{-1}\left(\|u\e_N-\kappa_B\e\|^2_{L_2(B\e)}+\eps^2\|\nabla u\e_N\|^2_{L_2(B\e)}\right).
\end{gather*}
 % ------------- %
Using them we infer from \eqref{Nestim1} that
 % ------------- %
\begin{multline*}
\left|\eta\e_N[w_N\e,w_N\e]\right|^2\leq 4(a\e)^2\intl_{S\e}\left|(v\e_N)_+-(v\e_N)_-\right|^2\,\D s\cdot \intl_{S\e}\left|(w\e_N)_+-(w\e_N)_-\right|^2\D s\\\leq 8 {(a\e)^2|S\e|(\kappa_F\e-\kappa_B\e)^2}\intl_{S\e}\left(\left|(u\e_N)_+-\kappa_F\e\right|^2+\left|(u\e_N)_--\kappa_B\e\right|^2\right)\,\D s\\ \leq C\left(\|u_N\e-\kappa\e_F\|^2_{L_2(F\e)}+\|u\e_N-\kappa_B\e\|^2_{L_2(B\e)}+\eps^2\|\nabla u\e_N\|^2_{L_2(B\e\cup F\e)}\right),
\end{multline*}
 % ------------- %
and thus by virtue of \eqref{lambdaN_est1} and \eqref{uN_est4} we get
 % ------------- %
\begin{gather}\label{wN_est1}
\eta\e_N[w_N\e,w_N\e]\to 0\;\text{ as }\;\eps\to 0.
\end{gather}
 % ------------- %
Finally, it follows \eqref{vN_est1}, \eqref{represN}, and \eqref{wN_est1} that
 % ------------- %
\begin{gather}\label{finalN}
\lambda_2^N(\eps)=\eta\e_N[u_N\e,u_N\e]\sim \eta\e_N[v_N\e,v_N\e]\sim\beta\;\text{ as }\;\eps\to 0,
\end{gather}
 % ------------- %
which is nothing else than the claim of the lemma.
\end{proof}

Finally, one has to inspect the behaviour of the other eigenvalues involved in estimates of the gap edges in the limit $\eps\to 0$.

\begin{lemma}\label{lm4}
One has
 % ------------- %
\begin{gather}\label{convPhi}
\text{if }\varphi=\pi\text{, then }\;\;\lambda_1^\varphi(\eps)\to \alpha\;\text{ as }\;\eps\to 0,\\\label{convPhi0}
\text{if }\varphi=0\text{, then }\;\;\lambda_2^\varphi(\eps)\to \beta\;\text{ as }\;\eps\to 0.
\end{gather}
 % ------------- %
\end{lemma}

\begin{proof}
The proof of \eqref{convPhi0} is similar to the argument used to demonstrate \eqref{convN}. Specifically, we approximate the eigenfunction ${u}_{\varphi}^\eps$ of $\A_\varphi^{\eps}$ that corresponds to $\lambda_2^{\varphi}(\eps)$ and satisfies the conditions \eqref{normalizationN}--\eqref{normalizationN++}, with the index $N$ replaced by $\varphi$, by the function ${v}_{N}^\eps$ given by \eqref{vN}. It is clear that ${v}_{N}^\eps\in\mathrm{dom}(\eta\e_\varphi)$ if $\varphi=0$. The check of the asymptotic equality
 % ------------- %
\begin{gather*}
\lambda^{\varphi}_{2}(\eps)\sim\eta_\varphi\e[{v}_{N}^\eps,{v}_{N}^\eps],\ \varphi=0
\end{gather*}
 % ------------- %
repeats word-by-word the proof of \eqref{finalN}.

Consider next $\varphi=\pi$. By $u\e_{\varphi}$ we denote the eigenfunction of $\A_\varphi^{\eps}$ that corresponds to $\lambda_1^{\varphi}(\eps)$ and satisfies the conditions \eqref{normalizationD}--\eqref{normalizationD+}, with the index $D$ replaced by $\varphi$. Obviously, for $\varphi=\pi$ such an eigenfuction exists.

Using the Cauchy inequality, a standard trace inequality and the Poincar\'{e} inequality one can employ the following estimate: for any $v\in H^1(F)$ we have
 % ------------- %
\begin{multline}\label{ineq1}
\left|\langle v\rangle_{S_{\pm}}-\langle v\rangle_F\right|^2\leq
{1\over |S_\pm|}\|v-\langle v\rangle_F\|^2_{L_2(S_{\pm})}\\\leq C\left(\|v-\langle v\rangle_F\|^2_{L_2(F)}+\|\nabla v\|^2_{L_2(F)}\right)\leq
C_1\|\nabla v\|^2_{L_2(F)}.
\end{multline}
 % ------------- %
 Via a change of variables, $x\mapsto x\eps$, one can easily derive from \eqref{ineq1} the estimate
 % ------------- %
\begin{gather}\label{ineq1eps}
\left|\langle u\e_{\varphi}\rangle_{S\e_{\pm}}-\langle u\e_{\varphi}\rangle_{F\e}\right|^2\leq
C\eps^{2-n}\|\nabla u\e_{\varphi}\|^2_{L_2(F\e)}.
\end{gather}
 % ------------- %
Furthermore, from \eqref{ineq1eps} and the fact that $u\e_{\varphi}$ satisfies condition \eqref{theta1} we can infer that
 % ------------- %
\begin{multline}\label{ineq4}
\left|\langle
u\e_{\varphi}\rangle_{F\e}\right|^2=\left|1-\exp(i\pi)\right|^{-2}\left|\langle
u\e_{\varphi}\rangle_{F\e}-\langle
u\e_{\varphi}\rangle_{S\e_+}+\exp(i\pi)\langle
u\e_{\varphi}\rangle_{S\e_-}-\exp(i\pi)\langle
u\e_{\varphi}\rangle_{F\e}\right|^2\\\leq
2\left|1-\exp(i\pi)\right|^{-2}\left(\left|\langle
u\e_{\varphi}\rangle_{F\e}-\langle
u\e_{\varphi}\rangle_{S_{+}\e}\right|^2+ \left|\langle
u\e_{\varphi}\rangle_{S_{-}\e}-\langle
u\e_{\varphi}\rangle_{F\e}\right|^2\right)\\\leq C\eps^{2-n}\|\nabla
u\e_{\varphi}\|^2_{L_2(F\e)}.
\end{multline}
 % ------------- %
It follows from  \eqref{ineq4} and the Poincar\'{e} inequality that
 % ------------- %
\begin{gather}\label{ineq5}
\|u\e_{\varphi}\|^2_{L_2(F\e)}=\|u\e_{\varphi}-\langle
u\e_{\varphi}\rangle_{F\e}\|^2_{L_2(F\e)}+\left|\langle
u\e_{\varphi}\rangle_{F\e}\right|^2\cdot |F\e|\leq C\eps^2\|\nabla
u\e_{\varphi}\|^2_{L_2(F\e)}.
\end{gather}
 % ------------- %
This means that similarly to the Dirichlet eigenfunction case the function $u\e_{\varphi}$ satisfies the Friedrichs inequality in $F\e$, despite the fact that $u\e_{\varphi}$ does not vanish on $\partial Y$. The remaining part of the argument leading to \eqref{convPhi} repeats literarily the proof of \eqref{convD}: we approximate the eigenfunction $u\e_{\varphi}$ by the function $v\e_{D}$ \eqref{vD}, noting that since $v\e_D$ vanishes in the vicinity of $\partial Y\e$ it belongs to $\mathrm{dom}(\eta_\varphi^\eps)$ for an arbitrary $\varphi$, and then check the asymptotic equality
 % ------------- %
\begin{gather}\label{asy}
\lambda^{\phi}_{1}(\eps)\sim \eta_\varphi\e[v_D\e,v_D\e],\quad \varphi=\pi.
\end{gather}
 % ------------- %
The proof of \eqref{asy} repeats word-by-word that of \eqref{finalD} taking into account the inequality \eqref{ineq5}.
\end{proof}

\addtocounter{remark}{1}
\begin{remark}
With some slight modifications of the proof one is able to show that \eqref{convPhi} holds in fact for all values of the parameter $\varphi\not= 0$. For our purposes, however, it is sufficient to consider antiperiodic (i.e. $\varphi=\pi$) eigenvalues only.
\end{remark}

Now, with the preliminaries represented by Lemmata~\ref{lm1}--\ref{lm4} it is not difficult to prove Theorem \ref{th1}. Indeed, it follows from \eqref{repres1} that
 % ------------- %
\begin{gather}\label{sp}
\sigma(\mathcal{A}\e)=\cupl_{k=1}^\infty [\lambda_k^-(\eps),\lambda_k^+(\eps)],
\end{gather}
 % ------------- %
where the spectral bands, i.e. the compact intervals $[\lambda_k^-(\eps),\lambda_k^+(\eps)]$, are defined by
 % ------------- %
\begin{gather}\label{interval}
[\lambda_k^-(\eps),\lambda_k^+(\eps)]= \cupl_{\phi\in [0,2\pi)}
\left\{\lambda_k^{\varphi}(\eps)\right\}.
\end{gather}
 % ------------- %
We set $\varphi_-=0$, $\varphi_+=\pi$. It follows from \eqref{enclosure} and \eqref{interval} that
 % ------------- %
\begin{gather}\label{double1}
\lambda_k^{N}(\eps)\leq \lambda_k^-(\eps)\leq
\lambda_k^{\varphi_-}(\eps),\\\label{double2}
\lambda_k^{\varphi_+}(\eps)\leq \lambda_k^+(\eps)\leq
\lambda_k^{D}(\eps).
\end{gather}
 % ------------- %
Obviously, the left- and right-hand-sides of \eqref{double1} are equal to zero if $k=1$. It follows from \eqref{convN}, \eqref{convPhi0} that in the case $k=2$ they both converge to $\beta$ as $\eps\to 0$, and consequently
 % ------------- %
\begin{gather}\label{a-}
\lambda_1^-(\eps)=0,\quad \liml_{\eps\to 0}\lambda_2^-(\eps)=\beta.
\end{gather}
 % ------------- %
Similarly, in view of \eqref{Phi2}, \eqref{convD}, and \eqref{convPhi} we obtain
 % ------------- %
\begin{gather}\label{a+}
\liml_{\eps\to 0}\lambda_1^+(\eps)=\alpha,\quad  \liml_{\eps\to
0}\lambda_2^+(\eps)=\infty.
\end{gather}
 % ------------- %
Then the relations \eqref{main1}--\eqref{main2} follow directly from \eqref{sp} and \eqref{a-}--\eqref{a+} in combination with the monotonicity of the sequences $\{\lambda_k^\pm(\eps)\}_{k\in\mathbb{N}}$. In this way, Theorem \ref{th1} is proved.

%=========================================================================
\section*{Acknowledgements} The research was supported by the German Research Foundation through GRK 1294 ``Analysis, Simulation and Design of Nanotechnological Processes'' and by Czech Science Foundation (GA\v{C}R) within the project 14-06818S.

%=========================================================================
%=========Bibliography====================================================
%=========================================================================

\end{document}